\documentclass[12pt,reqno]{amsart}
\usepackage{graphicx,indentfirst}
\usepackage{amsmath,amssymb,mathrsfs}
\usepackage{amsthm,amscd}
\usepackage{verbatim}
\usepackage{appendix}
\usepackage{enumitem,titletoc}
\usepackage{appendix}
\usepackage[utf8]{inputenc}
\usepackage{imakeidx}
\makeindex[columns=2, title=Alphabetical Index]
\usepackage{fancyhdr}
\usepackage{amsfonts,color}
\usepackage[all]{xy}
\usepackage{tikz-cd}
\usepackage{syntonly}
\usepackage{fancyhdr}
\newtheorem{theorem}{Theorem}
\newtheorem{lemma}{Lemma}
\newtheorem{remark}{Remark}
\usepackage[left=2.5cm,right=2.5cm,bottom=3cm,top=3cm]{geometry}

\usepackage{tikz}
\usepackage{extarrows}
\usepackage{hyperref}

\def\XXint#1#2#3{{\setbox0=\hbox{$#1{#2#3}{\int}$ }
\vcenter{\hbox{$#2#3$ }}\kern-.6\wd0}}

\newtheorem{defn}{Definition}

\allowdisplaybreaks
\newcommand{\del}{\partial}
\newcommand{\dbar}{\bar\partial}
\newcommand{\eps}{\varepsilon}
\newcommand{\ddbar}{\sqrt{-1}\partial\bar\partial}

\newcommand{\lapl}{\Delta}

\makeindex

\DeclareMathOperator{\ric}{Ric}
\DeclareMathOperator{\wg}{\wedge}
\DeclareMathOperator{\im}{\sqrt{-1}}

\DeclareMathOperator{\re}{\text{Re}}

\DeclareMathOperator{\tr}{tr}

\allowdisplaybreaks

\hypersetup{
    colorlinks=true,
    citecolor=green,
    filecolor=green,
    linkcolor=blue,
    urlcolor=black
}

\numberwithin{equation}{section}

\title{On the Hessian-cscK equations}
\author{Bin Guo, Kevin Smith and Freid Tong}
\date{}

\begin{document}

\maketitle
\begin{abstract}
In this paper, we propose a coupled system of complex Hessian equations which generalizes the equation for constant scalar curvature K\"ahler (cscK) metrics. We show this system can be realized variationally as the Euler-Lagrange equation of a Hessian version of the Mabuchi K-energy in an infinite dimensional space of $k$-Hessian potentials, which can be seen as an infinite dimensional Riemannian manifold with negative sectional curvature. Finally, we prove an a priori $C^0$-estimate for this system which depends on  the Entropy, which generalizes a fundamental result of Chen and Cheng \cite{CC} for cscK metrics. 
\end{abstract}

\section{Introduction}

There has been increasing interest in recent years in two directions in the theory of geometric partial differential equations: on one hand, on systems consisting of a non-linear equation coupled with its linearization, of which the constant scalar curvature K\"ahler metric (cscK) is an example in complex geometry \cite{CC}, and the affine Plateau problem an example from real geometry \cite{TW}; on the other hand, consideration of other elliptic equations besides the most familiar examples of the Laplacian and the Monge-Amp\`ere equations (e.g. Harvey and Lawson on the Lagrangian equation \cite{HL}, Collins-Yau on the defomed HYM equation \cite{CY}, Phong-Picard-Zhang on the Fu-Yau equation \cite{PPZ}). In this paper, we consider a coupled system which stands at the crossroads of the above two broad lines of development, namely the coupled system of a complex Hessian equation with its linearization. We show that this system admits a natural interpretation in terms of a generalized notion of curvature, and that this notion of curvature admits, just as the standard notion, an interpretation in terms of Deligne pairings \cite{PS0}. We establish the $C^0$ estimate for this Hessian coupled system, and give an interpretation of this coupled system as a variational problem for an energy functional in an infinite dimensional Riemannian manifold of negative sectional curvature, generalizing the constructions of Donaldson \cite{D}, Mabuchi \cite{M}, and Semmes \cite{S}.

We now describe the coupled system we are interested in. Let $(X,\omega)$ be a compact K\"ahler manifold. We consider the following coupled system of equations for a pair of smooth functions $(\varphi, F)$:
\begin{equation}\label{eqn:c sigma}
\left\{\begin{aligned}
&(\omega+ \ddbar \varphi)^k \wedge \omega^{n-k} = e^F\omega^n,\quad \sup_X\varphi = 0\\
&\lapl_G F = - \overline{\alpha} + \tr_G \alpha,
\end{aligned}\right.
\end{equation}
where $G^{i\bar j} = k\frac{\im dz^i\wedge d\bar z^j\wedge\omega_\varphi^{k-1}\wg\omega^{n-k}}{\omega_\varphi^k\wg\omega^{n-k}}$ and $\lapl_G = G^{i\bar j}\del_i\del_{\bar j}$ is the linearized operator associated to the nonlinear operator $\varphi\mapsto \log \frac{\omega_{\varphi}^k\wg\omega^{n-k}}{\omega^n}$.  $\alpha$ is a smooth $(1,1)$-form and $\overline{\alpha} = \frac{1}{V} \int_X \alpha\wedge \omega_\varphi^{k-1}\wedge \omega^{n-k}$ is a constant making the second equation of \eqref{eqn:c sigma} comptible.
For this system to be elliptic (i.e. $G^{i\bar j}>0$), we require $\varphi$ to be {\em admissible}, which means it satisfies the condition
\[(\omega+ \ddbar \varphi)^j \wedge \omega^{n-j} > 0 \text{ for } j = 1,\ldots, k.\] 
Alternatively, $\varphi$ is admissible if and only if the eigenvalues of $\omega+\ddbar\varphi$ with respect to the K\"ahler metric $\omega$ is in the $\Gamma_k$-cone, where $\Gamma_k\subset \mathbb R^n$ is given by
\[\Gamma_k = \{\lambda\in \mathbb{R}^n: \sigma_1(\lambda)>0, \ldots, \sigma_k(\lambda)> 0\}\]
and $\sigma_k$ is the elementary symmetric polynomial of degree $k$ on $\mathbb{R}^n$. It is important to note that in general, the condition of being admissible for $k<n$ will depend on the background K\"ahler metric $\omega$, in particular, it may not be invariant under biholomorphic maps. 

We state some basic well-known properties of the cone $\Gamma_k$ which will be used later. For a more thorough description of the properties of $\Gamma_k$, we refer the readers to \cite{Wang}. 

\begin{lemma}\label{lem: Gamma_k-basics}
For $\lambda \in \Gamma_k$, we have
\begin{enumerate}
    \item $\sigma_{k-1, i}(\lambda)>0$ for any $1\leq i\leq n$, where $\sigma_{k-1, i}(\lambda)=\sigma_{k-1}(\lambda_1, \ldots, \lambda_{i-1}, \lambda_{i+1}, \ldots, \lambda_n)$. 
    \item (Garding's inequality) For any $\mu,\lambda\in \Gamma_k$, we have the inequality
    \[\sum_{i=1}^n \mu_i \frac{\partial \sigma_k(\lambda)}{\partial \lambda_i} \ge C(n,k) \sigma_k(\mu)^{\frac 1 k} \sigma_k(\lambda)^{\frac{k - 1}{k}}.\]
    for some explicit constant $C(n, k)$. 
\end{enumerate}
\end{lemma}

\section{Energy functionals}
A key aspect about the cscK equation is that the equation for a cscK metric can be realized as the Euler-Lagrange equation of a Mabuchi $K$-energy in the space of K\"ahler potentials. This is true for the system \eqref{eqn:c sigma} as well. In this section, we will introduce the analogue of the Mabuchi $K$-energy and show that it is critical points correspond precisely the solutions of \eqref{eqn:c sigma}. 

\begin{defn}
We define the (generalized) {\em Ricci curvature} of $\omega_\varphi$ (relative to $\omega$) as
$$\widehat{\ric}(\omega_\varphi) = - \ddbar \log \omega_\varphi^k\wedge \omega^{n-k} = -\ddbar f + \ric(\omega)$$
where $f = \log \frac{\omega_\varphi^k\wedge \omega^{n-k}}{\omega^n}$ and $\ric(\omega )$ is the usual Ricci curvature of $\omega$. 
\end{defn}
The generalized Ricci curvature of $\omega_\varphi$ lies in the first Chern class $c_1(X)$. In fact, the equation \eqref{eqn:c sigma} can be viewed as an equation for the generalized Ricci curvature $\widehat \ric(\omega_\varphi)$. Indeed \eqref{eqn:c sigma} is equivalent to the equation
\begin{equation}\label{eqn:main 1}
\tr_G( \widehat {\ric}(\omega_\varphi) - \ric(\omega) + \alpha  ) = \bar \alpha,
\end{equation}

Viewing $\tr_G \widehat {\ric}(\omega_\varphi)$ as the generalized Scalar curvature, we can define the analogue of the Mabuchi energy, whose critical points are precisely potentials of constant generalized Scalar curvature. We will write $SH_k(X,\omega)$ the set of functions $\varphi$ such that $\omega^{-1}\cdot \omega_\varphi\in \Gamma_k$ on $X$.

\begin{defn}
Given a constant $\lambda\in \mathbb R$, we define the Hessian Mabuchi energy $\mu_k$ by its variation: for a family of $\phi_t\in SH_k(X,\omega)$, we have
\begin{equation}\label{eqn:d mu}\frac{d}{dt} \mu_k(\phi_t) = - \frac{k}{V} \int_X \dot \phi_t (\widehat{\ric}(\omega_{\phi_t}) - \lambda \omega_{\phi_t}) \wedge\omega_{\phi_t}^{k-1}\wedge \omega^{n-k}.\end{equation}
Since we can always add a constant to $\mu_k$ without changing the variation, we will often also choose a normalization so that $\mu_k(0) = 0$. 
\end{defn}

It's clear from this definition that the critical points of $\mu_k$ is precisely the solution of equation~\ref{eqn:main 1} with $\alpha = \ric(\omega)$, which is exactly those $\omega_{\varphi}$ with constant generalized Scalar curvature. However, it is not clear from this definition that $\mu_k$ is well-defined. The following Theorem shows that $\mu_k(\cdot)$ is well-defined and gives and explicit formula for $\mu_k$. 
\begin{theorem}
$\mu_k(\varphi)$ can be expressed as follows:
\begin{equation}\label{eqn:mu}
\begin{split}
\mu_k(\varphi) = \frac 1 V &\int_X \Big(\log \frac{\omega_\varphi^k\wedge \omega^{n-k}}{\omega^n} + \lambda \varphi \Big) \omega_\varphi^k\wedge \omega^{n-k} - \frac{\lambda}{V(k+1)}\Big(\int_X \varphi \sum_{j=0}^{k} \omega_{\varphi}^{k-j}\wedge \omega^{n-k+j} \Big) \\
&  - \frac 1 V \int_X \varphi \sum_{j=1}^k(\ric(\omega) - \lambda\omega)\wedge \omega_\varphi^{k-j}\wedge \omega^{n-k+j - 1} . 
\end{split}
\end{equation}
\end{theorem}
We remark that the first integral in \eqref{eqn:mu} corresponds to the ``entropy term'', the second one to the usual $J$-functional, and the last one to the $J_{\ric(\omega) - \lambda\omega}$-functional in Kahler geometry.
\begin{proof}
It suffices for us to show the variation of $\mu_k$ is given by the formula~\ref{eqn:d mu}. For simplicity we omit the subscript $t$ in $\phi_t$ and write $\phi = \phi_t$, then we compute the variation of $\mu_k$ as defined above
    \begin{equation*}
\begin{split}
    \frac{d}{dt}\mu_k(\phi_t) &= \frac{1}{V}\int_X(\lapl_{G}\dot{\phi}+\lambda\dot{\phi})\,\omega_\phi^k\wedge \omega^{n-k} 
    + \frac 1 V \int_X \lapl_G\dot{\phi}\Big(\log \frac{\omega_\phi^k\wedge \omega^{n-k}}{\omega^n} + \lambda \phi \Big) \omega_\phi^k\wedge \omega^{n-k}\\
    & \qquad 
    -\frac{1}{V}\left(\int_X\lambda\dot{\phi} \omega_{\phi}^{k}\wedge \omega^{n-k}\right)
    -\frac 1 V \int_X\dot{\phi}\sum_{j=1}^k(\ric(\omega) - \lambda\omega)\wedge \omega_\phi^{k-j}\wedge \omega^{n-k+j - 1}\\
    & \qquad  -\frac 1 V \int_X\dot{\phi}(\ric(\omega) - \lambda\omega)\wedge\left(k\omega_{\phi}^{k-1}\wedge\omega^{n-k}-\sum_{j=1}^{k} \omega_\phi^{k-j}\wedge \omega^{n-k+j - 1}\right)\\
    &= \frac k V \int_X \dot{\phi}\Big(\ric(\omega)-\widehat{\ric}(\omega_{\phi}) + \lambda \ddbar \phi \Big)\wedge \omega_\phi^{k-1}\wedge \omega^{n-k}\\
    &\qquad -\frac k V \int_X\dot{\phi}(\ric(\omega) - \lambda\omega)\wedge\omega_{\phi}^{k-1}\wedge\omega^{n-k}\\
    &= -\frac k V \int_X \dot{\phi}\Big(\widehat{\ric}(\omega_{\phi})-\lambda\omega_{\phi}\Big)\wedge \omega_\phi^{k-1}\wedge \omega^{n-k} 
\end{split}
\end{equation*}
\end{proof}

Given the definition of $\frac{d}{dt} \mu_k$ in \eqref{eqn:d mu}, we can also compute the second variation of $\mu_k$, which is given by
\begin{equation}
\begin{split}
    \frac{d^2}{dt^2}\mu_k(\varphi_t) &= -\frac k V \int_X \ddot{\varphi}\Big(\widehat{\ric}(\omega_{\varphi})-\lambda\omega_{\varphi}\Big)\wedge \omega_\varphi^{k-1}\wedge \omega^{n-k} -\frac k V \int_X \dot{\varphi}\Big(-\ddbar\lapl_G\dot{\varphi}-\lambda\ddbar\dot{\varphi}\Big)\wedge \\
   &\qquad \wedge \omega_\varphi^{k-1}\wedge \omega^{n-k}
     -\frac{k(k-1)}{V} \int_X \dot{\varphi}\Big(\widehat{\ric}(\omega_{\varphi})-\lambda\omega_{\varphi}\Big)\wedge \ddbar\dot{\varphi}\wedge\omega_\varphi^{k-2}\wedge \omega^{n-k} \\
    & = -\frac k V \int_X \ddot{\varphi}\Big(\widehat{\ric}(\omega_{\varphi})-\lambda\omega_{\varphi}\Big)\wedge \omega_\varphi^{k-1}\wedge \omega^{n-k} 
    +\frac 1 V \int_X \lapl_G\dot{\varphi}\left(\lapl_{G}\dot{\varphi}+\lambda\dot{\varphi}\right)\omega_{\varphi}^{k}\wedge\omega^{n-k}\\
    &\qquad -\frac{k(k-1)} V \int_X \dot{\varphi}\Big(\widehat{\ric}(\omega_{\varphi})-\lambda\omega_{\varphi}\Big)\wedge \ddbar\dot{\varphi}\wedge\omega_\varphi^{k-2}\wedge \omega^{n-k} \\
    &= -\frac k V \int_X (\ddot{\varphi}-|\del\dot{\varphi}|^2_G)\Big(\widehat{\ric}(\omega_{\varphi})-\lambda\omega_{\varphi}\Big)\wedge \omega_\varphi^{k-1}\wedge \omega^{n-k} 
    +\frac 1 V \int_X \left|\lapl_{G}\dot{\varphi}\right|^2\omega_{\varphi}^{k}\wedge\omega^{n-k}\\
    &\qquad +\frac{k(k-1)} V \int_X \im\del\dot{\varphi}\wedge\dbar\dot{\varphi}\wedge\widehat{\ric}(\omega_{\varphi})\wedge\omega_\varphi^{k-2}\wedge \omega^{n-k} 
    -\frac k V \int_X |\del\dot{\varphi}|_G^2\widehat{\ric}_{\varphi}\wedge \omega_\varphi^{k-1}\wedge \omega^{n-k} 
\end{split}
\end{equation}

\begin{remark}
More generally, we can also consider the $\alpha$-twisted Hessian Mabuchi energy (with $\lambda = \bar \alpha/k$)
\begin{equation}\begin{split}\mu_{\alpha,k}(\varphi) & = \mu_k(\varphi) + \frac k V \int_0^1 \int_X \dot\varphi_s (\ric(\omega) - \alpha  )\wedge \omega_{\varphi_s}^{k-1}\wedge \omega^{n-k}\\
& =\frac 1 V \int_X \Big(\log \frac{\omega_\varphi^k\wedge \omega^{n-k}}{\omega^n} + \lambda \varphi \Big) \omega_\varphi^k\wedge \omega^{n-k} - \frac{\lambda}{V(k+1)}\Big(\int_X \varphi \sum_{j=0}^{k} \omega_{\varphi}^{k-j}\wedge \omega^{n-k+j} \Big) \\
 & \qquad  - \frac 1 V \int_X \varphi \sum_{j=1}^k(\alpha - \lambda\omega)\wedge \omega_\varphi^{k-j}\wedge \omega^{n-k+j - 1} . 
\end{split}\end{equation}
whose variation is given by
$$\frac{d}{dt} \mu_{k,\alpha}(\phi_t) = - \frac{k}{V}\int_X \dot\phi_t ( \widehat{\ric}(\omega_{\phi_t}) - \lambda \omega_{\phi_t} - \ric(\omega) + \alpha  )\wedge \omega_{\phi_t} ^{k-1}\wedge \omega^{n-k}.$$
It follows that an $\omega_\varphi$ satisfying \eqref{eqn:main 1} is a critical point of $\mu_{k,\alpha}$.
\end{remark}

\subsection{The Deligne pairing and the energy $\mu_k$} Similar to the interpretation of the Mabuchi $K$-energy as the metric of some line bundle from the Deligne pairing, as shown in Phong and Sturm \cite{PS0,PS} (see also \cite{PRS}), we explain in this section that when $X$ is a projective manifold, the energy functional $\mu_k(\cdot)$ in \eqref{eqn:d mu} can be regarded as the metric on some  $\mathbb R$-line bundle from the Deligne pairing. Suppose $L_0,\ldots, L_n$ are holomorphic Hermitian line bundles on $X$, then the Deligne pairing 
$$\langle L_0,\ldots, L_n \rangle$$ is a Hermitian line bundle over a point (for the precise definition, we refer to \cite{PS0}). The change of metric formula (c.f. (2.10) in \cite{PS}) states that 
$$\langle L_0\otimes {\mathcal O}(\phi_0),\ldots, L_n\otimes \mathcal O(\phi_n) \rangle= \langle L_0,\ldots, L_n \rangle \otimes \mathcal O(E)$$ where $\mathcal O(f)$ denotes the trivial line bundle equipped with the Hermitian metric $h e^{-f}$, and $E$ is given by 
\begin{equation}\label{eqn:metric formula}E = \int_X \sum_{j=0}^n \phi_j \wedge_{k<j} c_1(L_k \otimes \mathcal O(\phi_k)) \wedge \wedge_{j<k\le n} c_1(L_k)\end{equation}
and $c_1 (L\otimes \mathcal O(\phi)) = c_1(L)+ \ddbar \phi$ is the Chern form of the Hermitian line bundle $L\otimes \mathcal O(\phi)$.

Let $L\to X$ be a positive line bundle equipped with a Hermitian metric $h$, such that $\omega = {\mathrm {Ric}}(h) = -\ddbar \log h\in c_1(L)$. Let $K_X$ be the canonical line bundle on $X$. The K\"ahler metric $\omega$ induces a metric $\frac{1}{\omega^n}$  on $K_X$. We define a metrized $\mathbb R$-line bundle 
$$\mathcal M_{k, h} = \langle K_X, L,\ldots, L  \rangle ^{1/V} \langle L,\ldots, L \rangle^{\lambda k/ V(1+k)} $$
where $L$ is given the metric $h$ and $K_X$ is equipped with the metric $\frac{1}{\omega^n}$. With the help of the formula \eqref{eqn:metric formula}, we can easily verify that 
\begin{equation} \label{eqn:mu k formula}\begin{split} & \langle K_X \otimes \mathcal O(\Theta), \underbrace{L_\varphi,\ldots,L_\varphi}_{k \text{ terms}}, L,\ldots, L  \rangle ^{1/V} \langle \underbrace {L_\varphi,\ldots,L_\varphi}_{k+1 \text{ terms}} L,\ldots, L \rangle^{\lambda k/ V(1+k)} 
\\ = & \langle K_X, L,\ldots, L  \rangle ^{1/V} \langle L,\ldots, L \rangle^{\lambda k/ V(1+k)}\otimes \mathcal O(\mu_k(\varphi)),\end{split}\end{equation}
where $\Theta = \log \frac{\omega_\varphi^k\wedge \omega^{n-k}}{\omega^n}\in C^\infty(X,\mathbb R)$, and $L_\varphi = L\otimes \mathcal O(\varphi)$ is the line bundle $L$ equipped with the Hermitian metric $h e^{-\varphi}$. From \eqref{eqn:mu k formula} and \eqref{eqn:mu} we see that the functional $\mu_k(\varphi)$ can be interprated as the change of a Hermitian metric on the Deligne pairing line bundle $\mathcal M_{k, h}$.

\subsection{Space of $k$-Hessian potentials}
Following \cite{M}, we can define a Riemannian structure on the space of $k$-potentials and investigate its geometry. For simplicity we will assume in this section that $\int_X\omega^n = 1$. Set
\begin{equation}
\mathcal H_k(X, \omega) = \{u\in C^\infty(X) \ | \ \omega_u^j\wedge\omega^{n-j} > 0 \text{ for } j = 1, \ldots, k\}.
\end{equation}
This is an open set of $C^{\infty}(X)$, hence we can identify the tangent space $T_u\mathcal H_k$ with the space of smooth functions on $X$. Let us define an inner product of two tangent vectors $\varphi,\psi\in T_u\mathcal H_k$ by
\begin{equation}\label{eqn:Metric on H_k}
\langle\varphi,\psi\rangle = \int_X\varphi\,\psi\,\omega_u^k\wedge\omega^{n-k}.
\end{equation}
With this inner product, we can formally view $\mathcal H_k$ as an infinite dimensional Riemannian manifold. 
\subsection{Geodesic equation}
Now we want to define a connection $D$ that is compatible with (\ref{eqn:Metric on H_k}). This means that if $\varphi$, $\psi$ are tangent vector fields along a curve $u_t\in\mathcal H_k$, we should require
\begin{equation}\label{eqn:Time derivative of inner product 1}
\frac{d}{dt}\langle\varphi,\psi\rangle = \langle D_{\dot u}\varphi,\psi\rangle + \langle\varphi,D_{\dot u}\psi\rangle.
\end{equation}
Integrating by parts shows that
\begin{equation}\label{eqn:Time derivative of inner product 2}
\frac{d}{dt}\langle\varphi,\psi\rangle = \int_X\bigg\{\bigg(\dot\varphi - \frac12G^{p\bar q}(\dot u_{\bar q}\varphi_p + \varphi_{\bar q}\dot u_p)\bigg)\psi + \varphi\bigg(\dot\psi - \frac12G^{p\bar q}(\dot u_{\bar q}\psi_p + \psi_{\bar q}\dot u_p)\bigg)\bigg\}\omega_u^k\wedge\omega^{n-k}.
\end{equation}
Combining (\ref{eqn:Time derivative of inner product 1}) and (\ref{eqn:Time derivative of inner product 2}) motivates the following definition of a connection on $\mathcal H_k(X, \omega)$:
\begin{equation}\label{eqn:Connection on space of k-potentials}
D_{\dot u}\varphi = \dot\varphi - \frac12G^{p\bar q}(\dot u_{\bar q}\varphi_p + \varphi_{\bar q}\dot u_p).
\end{equation}
Thus we see that the equation $D_{\dot u}\dot u = 0$ for a geodesic in this setting is
\begin{equation}\label{eq: geodesics-1}
\ddot u - G^{p\bar q}\dot u_{\bar q}\dot u_p = 0.
\end{equation}
Similar to \cite{M, S, D}, we can recast this as a degenerate Hessian equation on a product manifold of dimension $n+1$. First we complexify the time variable $t$ by adding an imaginary part, and assume everything is independent of the imaginary part of $t$. Then multiplying equation~\eqref{eq: geodesics-1} by the term $\sqrt{-1}dt\wedge d\bar t \wedge \omega_{u}^{k}\wedge\omega^{n-k}$, we get
\[(u_{t\bar t} - G^{p\bar q} u_{t\bar q}u_{p\bar t}) \sqrt{-1}dt\wedge d\bar t \wedge \omega_{u}^{k}\wedge\omega^{n-k}= 0.\]
which is equivalent to the following degenerate, homogenous Hessian equation on the product manifold $M\times \mathbb{C}_t$
\begin{equation}\label{eq: geodesics-2}
(\omega+\ddbar u)^{k+1}\wedge\omega^{n-k} = 0. 
\end{equation}
Thus equation for geodesics segments joining two potentials $\varphi_0, \varphi_1$ can be formulated as the following boundary value problem for a homogenous degenerate Hessian equation. 
\begin{equation}\label{eqn: geod}
\begin{cases}
    ({\omega}+\ddbar u)^{k+1}\wedge {\omega}^{n-k} = 0\\
    u(\cdot, t) = \varphi_0 \text{ for } \re t = 0\\
    u(\cdot, t) = \varphi_1 \text{ for } \re t = 1 
\end{cases}
\end{equation}

There are some major analytic difficulties for solving the geodesic equation arising from the fact that it is degenerate in two different ways: not only is data on the right hand side of the equation is zero, but also the $(1, 1)$-form $\omega$ is also degenerate in the $t$ direction. This causes a lot of analytic difficulties, and we plan to investigate the existence and regularity of geodesics for this system in a subsequent work. 

\subsection{Curvature of $\mathcal H_k$}
In the rest of this section, we compute the curvature of $D$ and prove the following theorem. 
\begin{theorem}\label{thm:nonpositive-curvature}
    The sectional curvature of $D$ on $\mathcal H_k$ is non-positive. 
\end{theorem}
\begin{proof}
First we write the connection on $\mathcal H_k$ as 
\[D_{\dot u}\varphi = \dot\varphi - Q(\nabla\dot u, \nabla \varphi).
\]
where \[Q(\nabla \dot u, \nabla \psi) = \frac12G^{p\bar q}(\dot u_{\bar q}\varphi_p + \varphi_{\bar q}\dot u_p) = \frac k 2\frac{\sqrt{-1}(\partial\dot{u}\wedge\bar{\partial}\varphi+\partial \varphi\wedge \bar{\partial}\dot u)\wedge\omega_u^{k-1}\wedge\omega^{n-k}}{\omega_u^{k}\wedge\omega^{n-k}}.\] 
Suppose that $u(x,t,s)$ is a family of $k$-potentials, and let $R$ be the curvature of $D$, then we compute $R(u_t, u_s)\eta$,
\begin{align*}
    R(u_t, u_s)\eta &= D_t(\eta_s-Q(\nabla \eta, \nabla  u_s))-D_s(\eta_t-Q(\nabla \eta, \nabla  u_t))\\
    & = -\frac{\partial}{\partial t}Q(\nabla\eta, \nabla  u_s)+\frac{\partial}{\partial s}Q(\nabla\eta, \nabla u_t)-Q(\nabla u_t, \nabla (\eta_s-Q(\nabla \eta, \nabla  u_s)))\\
    & \qquad + Q(\nabla u_s, \nabla (\eta_t-Q(\nabla \eta, \nabla  u_t)))\\
    & = \frac{k(k-1)}{2}\frac{\im(\del \eta\wg\dbar u_t+\del u_t\wg \dbar \eta)\wg\ddbar u_s\wg\omega_u^{k-2}\wg\omega^{n-k}}{\omega_u^k\wedge\omega^{n-k}}\\
    & \qquad -\frac{k(k-1)}{2}\frac{\im(\del \eta\wg\dbar u_s+\del u_s\wg \dbar \eta)\wg\ddbar u_t\wg\omega_u^{k-2}\wg\omega^{n-k}}{\omega_u^k\wedge\omega^{n-k}}\\
    &\qquad +Q(\nabla\eta, \nabla u_s)\lapl_Gu_t-Q(\nabla\eta, \nabla u_t)\lapl_Gu_s\\
    &\qquad +Q(\nabla u_t, \nabla Q(\nabla \eta, \nabla  u_s)) - Q(\nabla u_s, \nabla Q(\nabla \eta, \nabla  u_t))\\
\end{align*}

\begin{align*}
    \langle R(u_t, u_s)u_s, u_t\rangle &= \frac{k(k-1)}{2}\int_Xu_t\im(\del u_s\wg\dbar u_t+\del u_t\wg \dbar u_s)\wg\ddbar u_s\wg\omega_u^{k-2}\wg\omega^{n-k}\\
    & \qquad -k(k-1)\int_Xu_t\im\del u_s\wg \dbar u_s\wg\ddbar u_t\wg\omega_u^{k-2}\wg\omega^{n-k}\\
    &\qquad +\int_X\left(Q(\nabla u_s, \nabla u_t)^2-Q(\nabla u_s, \nabla u_s)Q(\nabla u_t, \nabla u_t)\right)\omega_u^{k}\wg\omega^{n-k}\\
    & = k(k-1)\int_X (\im\del u_s\wg\dbar u_s)\wg (\im \del u_t\wg \dbar u_t)\wg\omega_u^{k-2}\wg\omega^{n-k}\\
    &\qquad +\int_X\left(Q(\nabla u_s, \nabla u_t)^2-Q(\nabla u_s, \nabla u_s)Q(\nabla u_t, \nabla u_t)\right)\omega_u^{k}\wg\omega^{n-k}\\
\end{align*}

If we fix a normal coordinate where $g_{i\bar j} = \delta_{i\bar j}$, $u_{i\bar j} = (\lambda_i-1)\delta_{i\bar j}$ and $\del_{i}u_t = X_i$, $\del_iu_s = Y_i$, then we have
\begin{align*}
    \langle R(u_t, u_s)u_s, u_t\rangle & = k!(n-k)!\int_X\sum_{ij }\sigma_{k-2, ij}\frac 1 2(|X_i|^2|Y_j|^2+|X_j|^2|Y_i|^2-X_i\bar Y_i Y_j\bar X_j-Y_i\bar X_iX_j \bar Y_j)\\
    &\qquad +k!(n-k)!\int_X \sum_{ij }\frac{\sigma_{k-1, i}\sigma_{k-1, j}}{4\sigma_k}(X_i\bar Y_iX_j\bar Y_j+X_i\bar Y_i Y_j \bar X_j+Y_i\bar X_i X_j\bar Y_j+Y_i\bar X_iY_j\bar X_j)\\
    & \qquad -k!(n-k)!\int_X\sum_{ij }\frac{\sigma_{k-1, i}\sigma_{k-1, j}}{\sigma_k}|X_i|^2|Y_j|^2\\
    & = k!(n-k)!\int_X\sum_{ij }\sigma_{k-2, ij}(|X_i|^2|Y_j|^2-X_i\bar Y_i Y_j\bar X_j)\\
    &\qquad +k!(n-k)!\int_X \sum_{ij }\frac{\sigma_{k-1, i}\sigma_{k-1, j}}{2\sigma_k}(\re (X_i\bar Y_iX_j\bar Y_j)+X_i\bar Y_i Y_j \bar X_j)\\
    & \qquad -k!(n-k)!\int_X\sum_{ij }\frac{\sigma_{k-1, i}\sigma_{k-1, j}}{\sigma_k}|X_i|^2|Y_j|^2\\
    & = k!(n-k)!\int_X \sum_{ij}\left(\sigma_{k-2, ij}-\frac{\sigma_{k-1, i}\sigma_{k-1, j}}{\sigma_k}\right)(\frac 1 2\left(|X_i|^2|Y_j|^2+|Y_i|^2|X_j|^2)-X_i\bar Y_i Y_j\bar X_j\right)
\end{align*}
\begin{lemma}\label{lemma 2.2}
For any $i, j$, the terms \[\sigma_{k-2, ij}-\frac{\sigma_{k-1, i}\sigma_{k-1, j}}{\sigma_k}\]
are non-positive. 
\end{lemma}
\begin{proof}
The inequality we have to show is equivalent to \[\sigma_{k-2, ij}\sigma_k\leq \sigma_{k-1, i}\sigma_{k-1, j}.\] If we expand both sides, the terms on each side that contains either a multiple $\lambda_i$ or a multiple of $\lambda_j$ will cancel, and it suffices to look at only terms that doesn't contain $\lambda_i$ or $\lambda_j$. Therefore the inequality is equivalent to
\[\sigma_{k-2, ij}\sigma_{k, ij}\leq \sigma_{k-1, ij}^2\]
which follows from Newton's inequality. 
\end{proof}

We remark that Lemma \ref{lemma 2.2} is stronger than the well-known fact that $\log \sigma_k(\lambda)$ is concave in $\Gamma_k$. By Young's inequality, for each $i, j$ we have
\[\frac{1}{2}(|X_i|^2|Y_j|^2 +|X_j|^2|Y_i|^2)\geq |X_i||Y_j||Y_i||X_j|\]
Therefore each term in the sum 
\[\sum_{ij}\left(\sigma_{k-2, ij}-\frac{\sigma_{k-1, i}\sigma_{k-1, j}}{\sigma_k}\right)\left(\frac 1 2(|X_i|^2|Y_j|^2+|Y_i|^2|X_j|^2)- \re(X_i\bar Y_i Y_j\bar X_j)\right)\]
is non-positive, and hence the sectional curvature $R$ of the $L^2$ metric on $\mathcal H_k$ is non-positive, which proves Theorem~\ref{thm:nonpositive-curvature}. 
\end{proof}

\section{$C^0$-estimate}
In this section, we prove a $C^0$ estimate for the system \eqref{eqn:c sigma}  depending on a generalized entropy. The main result of this section is a direct counterpart of Theorem 5.1 in \cite{CC}.
\begin{theorem}\label{thm:1}
Let $(\varphi, F)$ be a smooth solution to \eqref{eqn:c sigma}, then a bound on $\int_X e^{\frac n k F} |F|\omega^n$ implies a bound for $\|F\|_\infty$ and $\|\varphi\|_\infty$.
\end{theorem}

Before we present the proof, let us remark that by a result from \cite{DK, GPT}, using just the first equation in \eqref{eqn:c sigma}, we can bound the $L^{\infty}$ norm of $\varphi$ if the $e^F$ is bounded in $L^p$ for $p>\frac n k$. However, this estimate fails when $p = \frac n k$. Our result can be seen as a refinement of their result for the coupled system. 

\subsection{Proof}
We begin with the following lemma which gives a lower bound on $\det G^{i\bar j}$.

\begin{lemma}\label{lemma 1}
There is a constant $C=C(n,k)$ such that 
$$\det (G^{i\bar j}) \ge C(n,k) \sigma_k^{-\frac n k}.$$
\end{lemma}
\begin{proof}We denote the eigenvalues of $\omega_\varphi$ (w.r.t. $\omega$) by $\lambda = (\lambda_1,\ldots, \lambda_n)\in \Gamma_k$. 
Then $G^{i\bar j} = \sigma_k(\lambda) ^{-1} \frac{\partial \sigma_k(\lambda)}{\partial \lambda_i} \delta_{ij}$. 
By Garding's inequality (Lemma~\ref{lem: Gamma_k-basics}) we know that if $\mu = (\mu_1,\ldots, \mu_n)\in \Gamma_k$, then 
$$\sum_{i=1}^n \mu_i \frac{\partial \sigma_k}{\partial \lambda_i} \ge C(n,k) \sigma_k(\mu)^{\frac 1 k} \sigma_k(\lambda)^{\frac{k - 1}{k}}.$$
Taking infimum over all $\mu\in \Gamma_n$ with $\prod_{i=1}^n \mu_i = 1$, we see the LHS of the above becomes $n (\prod_{i=1}^n \frac{\partial\sigma_k(\lambda)}{\partial \lambda_i} )^{1/n}$, and the RHS is bounded below by 
$$C(n,k) \sigma_n(\mu)^{\frac 1 n} \sigma_k(\lambda)^{\frac{k-1}{k}} = C(n,k) \sigma_k(\lambda)^{\frac{k-1}{k}}.$$
Combining these inequalities the lemma follows straightforwardly.

\end{proof}
As in \cite{CC}, we introduce an auxiliary complex Monge-Amp\`ere equation to prove the $C^0$ estimate. 
For notation convenience we denote $\Phi(F) = \sqrt{F^2 + 1}$ and $$A_F = \int_X e^{\frac n k F} \Phi(F) \omega^n. $$ $A_F$ is bounded by our assumption. We consider the complex MA equation
\begin{equation}\label{eqn:MA}
\left\{\begin{aligned}
& \omega_\psi^n = \frac{e^{\frac n k F} \Phi(F)}{A_F }\omega^n,\quad \sup_X\psi = 0\\
& \omega_\psi = \omega + \ddbar\psi >0,
\end{aligned}\right.
\end{equation}
which admits a unique solution by Yau's theorem \cite{Y}. The lemma below is the key step to obtain the $L^\infty$ estimate of $\varphi$, and the proof follows closely the ABP-type argument in \cite{CC} (see also \cite{GPT}).

\begin{lemma}\label{lemma 2}
For any $\varepsilon>0$, there exist constants $\lambda>0$ and $C = C(n, k,\varepsilon)>0$ such that
$$F + \varepsilon \psi - \lambda \varphi \le C.$$
\end{lemma}
\begin{proof}
For notation convenience we denote $\phi(t) = \phi_\delta(t) = t + \sqrt{t^2 + \delta}>0$ which converges to $\max(t, 0)$ as $\delta\to 0$. Here $\delta>0$ is a small number which will go to zero. We also denote $f : = F+\varepsilon \psi - \lambda \varphi$ for notation simplicity and we will look at $\phi(f)$ which converges to $f_+$ as $\delta\to 0$ and is a regularization of $2 f_+$.

\medskip

We define a {\em smooth} function 
$$H = \phi(f)^q,$$
where  $q = 1 + \frac 1{2n}>1$ is constant. Since $X$ is compact, we may assume $H$ achieves its maximum at $x_0 $ and $\max_X H = M>1$. Let $r = r(X,\omega)>0$ be the injectivity radius of $(X,\omega)$ as a Riemannian manifold. So we can identify the geodesic ball $B_r(x_0)$ as the Euclidean ball $B_{\mathbb C^n}(0,r)$ for simplicity. Let $\theta \in (0,1)$ be 
\begin{equation}\label{eqn:theta}\theta : = \min\{ \frac{r^2}{1000 M^{1/p}}, \frac{1}{100n} \}<\frac 1{10}.\end{equation}
Choose an auxiliary smooth function $\eta$ defined on $B_r(x_0)$ so that $\eta = 1$ on $B_{r/4}(x_0)$ and $ \eta = 1 -\theta$ on $B_{3r/4}(x_0)$, and $\eta$ also satisfies 
$$|\nabla \eta|_g^2 \le \frac{10 \theta^2}{r^2},\quad |\nabla^2 \eta|_g \le \frac{10\theta}{r^2},$$
where we identify $\omega$ with its associated Riemannian metric $g$.

\newcommand{\fpp}{F+\varepsilon \psi - \lambda \varphi}

We calculate
\begin{equation}\label{eqn:1}\Delta_G( H \eta ) = \eta \Delta_G H  + H \Delta_G \eta + 2 Re\big( G^{i\bar j} \nabla_{\bar j} H \nabla_i \eta   \big).\end{equation}
Observe that (below for a function $f$, $|\nabla f|_G^2 = G^{i\bar j} \nabla_{\bar j} f\nabla_i f$ )
$$H \Delta_G \eta = H \tr_G \ddbar \eta\ge - H \frac{10\theta}{r^2} \tr_G \omega.$$
And
\begin{align*}
Re\big( G^{i\bar j} \nabla_{\bar j} H \nabla_i \eta   \big) & = q \phi(f)^{q-1} Re \big( G^{i\bar j}\nabla_{\bar j} \phi(f) \nabla_i \eta   \big)\\
& \ge -\frac{q (q-1)}{4} \phi(f)^{q-2}| \nabla \phi(f)  |^2_G  - \frac{q}{q-1}\phi(f)^q  |\nabla \eta|^2 _G\\
& \ge -\frac{q (q-1)}{4} \phi(f)^{q-2}| \nabla \phi(f)  |^2_G - \frac{q}{q-1} \phi(f)^q  \frac{10\theta^2}{r^2} \tr_G \omega.
\end{align*}
And
\begin{align*}
\eta \Delta_G H & = q\eta\phi(f)^{q-1} \Delta_G \phi(f) + q (q-1) \eta \phi(f)^{q-2} | \nabla \phi(f)|_G^2\\
 & = q\eta\phi(f)^{q-1}  \phi'(f) \Delta_G f + q \eta \phi(f)^{q-1} \phi''(f) |\nabla f |^2_G    + q (q-1) \eta \phi(f)^{q-2} | \nabla \phi(f)|_G^2.
\end{align*}
We note that the middle term above is nonnegative due to the fact that $$\phi''(t) = \frac{1}{\sqrt{t^2 + \delta}} - \frac{t^2}{(\sqrt{t^2 + \delta})^3}>0 $$
For the first term we calculate
\begin{align*}
\Delta_G f = \Delta_G (\fpp)& = \tr_G \alpha - \overline{\alpha} + \varepsilon \tr_G \omega_\psi - \varepsilon \tr_G \omega - \lambda k + \lambda \tr_G \omega\\
& \ge (\lambda -\varepsilon) \tr_G \omega + \tr_G \alpha - C(n,k,\alpha) + \varepsilon n (\det G \cdot \det \omega_\psi)^{1/n}\\
& \ge (\lambda -\varepsilon) \tr_G \omega + \tr_G \alpha - C(n,k,\alpha) + \varepsilon n (c e^{-\frac n k F} \cdot e^{\frac n k F}\Phi(F) A_F^{-1})^{1/n}\\
& \ge (\lambda -\varepsilon) \tr_G \omega + \tr_G \alpha - C(n,k,\alpha) + c(n,k)\varepsilon \Phi(F)^{1/n} A_F^{-1/n},
\end{align*}
where in the second inequality we use Lemma \ref{lemma 1}.
Plugging these inequalities into \eqref{eqn:1}, if we choose $\lambda = 10 + \sup_X |\alpha|_\omega$
\begin{equation}\label{eqn:text}
\begin{split}
\Delta_G( H \eta ) & \ge - \frac {10 \theta}{r^2} \phi(f)^q  \tr_G \omega - \frac{q}{q-1} \phi(f)^q \frac{10\theta^2}{r^2} \tr_G\omega \\
&\qquad  + q \phi(f)^{q-1} \phi'(f) \Big( \tr_G \big( (\lambda - \varepsilon )\omega + \alpha   \big) + c(n,k) \varepsilon\Phi(F)^{1/n} A_F^{-1/n} - C   \Big) \\
& \ge  q \phi(f)^{q-1} \Big( 2 \phi'(f) \tr_G \omega - \frac{10\theta}{r^2 q}\phi(f) \tr_G \omega \\
&\qquad  - \frac{10\theta^2}{(q-1) r^2} \phi(f) \tr_G \omega + c(n,k) \phi'(f) \varepsilon\Phi(F)^{1/n} A_F^{-1/n} - C \phi'(f)   \Big).
\end{split}\end{equation}
To deal with RHS in the equation \eqref{eqn:text}, note that on the set $\{f\le 0\}$ $$\phi(f) = f + \sqrt{f^2 + \delta} = \frac{\delta}{\sqrt{f^2 + \delta} - f}\le \sqrt{\delta},$$
and $$1\ge \phi'(f) = 1 + \frac{f}{\sqrt{f^2 + \delta}} = \frac{\phi(f)}{\sqrt{f^2 + \delta}}\ge 0.$$
So the in the set $\{f\le 0\}$ RHS of \eqref{eqn:text} is
\begin{align*}
\ge q \phi(f)^{q-1} \Big( - \frac{10\theta}{r^2 q}\sqrt{\delta} \tr_G \omega  - \frac{10\theta^2}{(q-1) r^2}\sqrt{\delta} \tr_G \omega - C   \Big)
\end{align*}
On the other, on the set $\{f>0\}$, we know $\phi'(f)>1$, so the RHS of \eqref{eqn:text} is 
$$\ge q \phi(f)^{q-1} \big( c(n, k) \varepsilon \Phi(f)^{1/n} A_F^{-1/n} - C   \big),$$
where in the last inequality we use the choice of $\theta$ in \eqref{eqn:theta} and the fact that $\phi(f)\le M^{1/q}$.

Combining the above two cases, we obtain that 
\begin{align*}
\Delta_G(H\eta) & \ge   q \phi(f)^{q-1} \big( c(n, k) \varepsilon \Phi(f)^{1/n} A_F^{-1/n} - C   \big) \chi_{\{f>0\}} \\
&\qquad - q \phi(f)^{q-1} \Big(  \frac{10\theta}{r^2 q}\sqrt{\delta} \tr_G \omega  +\frac{10\theta^2}{(q-1) r^2}\sqrt{\delta} \tr_G \omega + C   \Big)\chi_{\{f\le 0\}},
\end{align*}
where $\chi_E$ denotes the characteristic function of a given set $E$.

We now apply the ABP maximum principle (\cite{GT}) on the (Euclidean) ball $B_r(x_0)$, and we get
{\small
\begin{align*}
\sup_{B_r(x_0)} (H\eta) & \le \sup_{\partial B_r(x_0)} (H \eta) + C(n) r\Big\{ \int_{B_r(x_0)\cap \{f>0\}} \frac{\big( \phi(f)^{q-1}  \big)^{2n}  \big( c(n,k) \varepsilon \Phi(F)^{1/n} A_F^{-1/n} - C    \big)_-^{2n} }{ (\det G)^2  }    \\
&\qquad + \int_{ B_r(x_0)\cap \{f\le 0\}   }  \frac{ (\phi(f)^{q-1})^{2n}  (  \frac{10\theta}{r^2 q}\sqrt{\delta} \tr_G \omega  +\frac{10\theta^2}{(q-1) r^2}\sqrt{\delta} \tr_G \omega + C   )^{2n}      }{(\det G)^2}         \Big\}^{1/2n}\\
& \le \sup_{\partial B_r(x_0)} (H \eta) + C(n) r\Big\{ \int_{B_r(x_0)\cap \{f>0\}} \frac{\big( \phi(f)^{q-1}  \big)^{2n}  \big( c(n,k) \varepsilon \Phi(F)^{1/n} A_F^{-1/n} - C    \big)_-^{2n} }{ e^{- 2 n F/k}   }    \\
&\qquad +C(n, F, G,\omega) \delta^{n(q-1)}        \Big\}^{1/2n}
\end{align*}
}
where the constant $C(n, F, G, \omega)$ is not uniformly bounded, but this is not a concern, since later on we will let $\delta\to 0$.
We observe that  the integral above is in fact integrated over the set where $ c(n,k)\varepsilon \Phi(F)^{1/n} A_F^{-1/n} - C <0$ and $\fpp >0$, and over this set $F\le C(n, k, A_F, \varepsilon)$ by the definition of $\Phi(F) = \sqrt{F^2 + 1}$. On the other hand, on this set $\fpp \le C  - \lambda \varphi$. Therefore, we have
\begin{align*}M = \sup_{B_r(x_0)} (H\eta) &  \le (1-\theta) \sup_{\partial B_r(x_0)} H + C\Big(\int_{B_r(x_0)}  ( C -\lambda \varphi) \omega^n + C(n, F, G,\omega) \delta^{n(q-1)}   \Big)^{1/2n}\\ & \le (1-\theta) M + C+C(n, F, G,\omega)  \delta^{(q-1)/2}
 \end{align*}
where we have use the inequality that $\int_X (-\varphi)\omega^n \le C$ which follows by the Green formula, $n + \Delta_\omega \varphi >0$ and the normalization condition $\sup_X \varphi = 0$.
Hence we conclude that
$$M^{1-\frac 1 q}\le C + C(n, F, G,\omega)  \delta^{(q-1)/2}
, \,\Rightarrow \, M \le C + C(n, F, G, \omega) \delta^{2 q},$$ which says that 
$$\sup_X  2 f_+ \le \sup_X \phi(f) \le C + C(n, F, G, \omega) \delta^{2 q}$$
letting $\delta\to 0$ gives the desired estimate.

\end{proof}

\begin{proof}[Proof of Theorem \ref{thm:1}]
\newcommand{\fpp}{F+\varepsilon \psi - \lambda \varphi}

Lemma \ref{lemma 2} shows that for any $p> n/k$, we have
$$e^{p F} = e^{p(\fpp) - p \varepsilon \psi + p \lambda \varphi}\le C e^{-p\varepsilon \psi},$$ if $\varepsilon$ is chosen small enough so that $p \varepsilon < \alpha(X,\omega)$, the $\alpha$-invariant of the Kahler manifold $(X,\omega)$. We can get  the $L^\infty$ bound of $\varphi$ by a result of Dinew-Kolodziej (\cite{DK}, see also \cite{GPT,GPT1}).  Kolodziej's $L^\infty$-estimate for complex MA equations implies the $L^\infty$-bound of $\psi$. Plugging these estimates to Lemma \ref{lemma 2} again, we get the upper bound of $F$, i.e. $F\le C$. 

To see the lower bound of $F$, we calculate (if we take $A = 1 + \sup |\alpha|_\omega$)
\begin{align*}
\Delta_G ( F  + A \varphi  ) & = \tr_G \alpha - \overline{\alpha} + A \tr_G \tilde\omega - A \tr_G \omega\\
& \le - \tr_G \omega + Ak - \overline{\alpha}\\
& \le - n ( \det G \cdot \det \omega )^{1/n} + A k - \overline\alpha\\
& \le - c(n,k) e^{-\frac{F}{k}} + Ak - \overline\alpha.
\end{align*}
At the minimum of $F+A\varphi$, we have $F>-C(n,k,\omega, \overline\alpha)$. Then the lower bound of $F$ follows easily from the $L^\infty$ estimate of $\varphi$.

\end{proof}

\smallskip

\noindent{\bf Acknowledgement:} 
We would like to thank Prof. D.H. Phong for many helpful suggestions, and for his continuous and generous support and encouragement. F.T. would like to thank Harvard CMSA for supporting his work.

\noindent Bin Guo\\
\noindent Department of Mathematics \& Computer Science, Rutgers University, Newark, NJ 07102\\
bguo@rutgers.edu

\medskip

\noindent Kevin Smith\\
\noindent Department of Mathematics, Columbia University, New York, NY 10027\\
kjs@math.columbia.edu

\medskip

\noindent Freid Tong\\
\noindent Center of Mathematical Sciences and Applications, Harvard University, Cambridge, MA 02138
\\
ftong@cmsa.fas.harvard.edu

\end{document}